\documentclass[11pt]{amsart}
\usepackage{amsfonts,latexsym,amsthm,amssymb,graphicx}
\usepackage[all]{xy}
\usepackage{hyperref}
\usepackage[usenames]{color} 

\DeclareFontFamily{OT1}{rsfs}{}
\DeclareFontShape{OT1}{rsfs}{n}{it}{<-> rsfs10}{}
\DeclareMathAlphabet{\mathscr}{OT1}{rsfs}{n}{it}

\CompileMatrices

\newtheorem{theorem}{Theorem}[section]
\newtheorem{lemma}[theorem]{Lemma}
\newtheorem{corol}[theorem]{Corollary}
\newtheorem{prop}[theorem]{Proposition}

{\theoremstyle{definition} \newtheorem{defin}[theorem]{Definition}}
{\theoremstyle{remark} \newtheorem{remark}[theorem]{Remark}
\newtheorem{example}[theorem]{Example}}

\numberwithin{equation}{section}

\newcommand{\Cbb}{{\mathbb{C}}}
\newcommand{\Pbb}{{\mathbb{P}}}
\newcommand{\Zbb}{{\mathbb{Z}}}
\newcommand{\cE}{{\mathscr E}}
\newcommand{\cI}{{\mathscr I}}
\newcommand{\cL}{{\mathscr L}}
\newcommand{\cM}{{\mathscr M}}
\newcommand{\cO}{{\mathscr O}}
\newcommand{\cV}{{\mathscr V}}
\newcommand{\cX}{{\mathscr X}}
\newcommand{\cY}{{\mathscr Y}}
\newcommand{\one}{1\hskip-3.5pt1}
\newcommand{\csm}{{c_{\text{SM}}}}
\newcommand{\cvir}{{c_{\text{vir}}}}
\newcommand{\ctop}{{c_{\text{top}}}}
\newcommand{\qede}{\hfill$\lrcorner$}

\DeclareMathOperator{\rk}{rk}
\DeclareMathOperator{\codim}{codim}
\DeclareMathOperator{\Spec}{Spec}
\DeclareMathOperator{\Sym}{Sym}

\title{
The Chern-Schwartz-MacPherson class of an embeddable scheme
}
\author{Paolo Aluffi}
\address{
Mathematics Department, 
Florida State University,
Tallahassee FL 32306, U.S.A.
}
\email{aluffi@math.fsu.edu}

\begin{document}

\begin{abstract}
The Chern-Schwartz-MacPherson class of a hypersurface in a nonsingular
variety may be computed directly from the Segre class of the jacobian
subscheme of the hypersurface; this has been known for a number of
years. We generalize this fact to arbitrary embeddable schemes: for
every subscheme $X$ of a nonsingular variety~$V$, we define an
associated subscheme $\cY$ of a projective bundle over $V$ and provide
an explicit formula for the Chern-Schwartz-MacPherson class of $X$ in
terms of the Segre class of~$\cY$.  If $X$ is a local complete
intersection, a version of the result yields a direct expression for
the Milnor class of $X$.

For $V=\Pbb^n$, we also obtain expressions for the
Chern-Schwartz-MacPherson class of~$X$ in terms of the `Segre zeta
function' of $\cY$.
\end{abstract}

\maketitle


\section{Introduction}\label{s:intro}
\subsection{}
The goal of this paper is the generalization to arbitrary subschemes
of nonsingular varieties of a twenty-year old formula for the {\em
  Chern-Schwartz-MacPherson\/} class of hypersurfaces, in terms of the
Segre class of an associated scheme. We first recall the general
context and the relevant definitions; the hurried reader may want to
skip ahead to~\S\ref{ss:projr} for the statement of the main result
for subschemes of projective space.

\subsection{}\label{ss:spat}
Every nonsingular variety $X$ has a canonically defined class in its
homology, namely the total Chern class of its tangent bundle. Deligne
and Grothendieck conjectured, and MacPherson proved
(\cite{MR0361141}), that (at least in characteristic $0$) this class
is a manifestation of a functorial theory of Chern classes which
assigns a distinguished homology class to {\em every\/} complex
projective variety~$X$. The class associated with $X$ is called the
{\em Chern-Schwartz-MacPherson\/} (CSM) class of $X$, $\csm(X)$.
(Brasselet and Schwartz proved (\cite{MR83h:32011}) that the class
$\csm(X)$ agrees via Alexander duality with the class defined earlier
by Marie-H\'el\`ene Schwartz, \cite{MR35:3707, MR32:1727}.)
MacPherson's theory can be refined to give a class in the Chow group
of $X$ (\cite[Example~19.1.7]{85k:14004}), and extended to embeddable
schemes over arbitrary algebraically closed fields of
characteristic~$0$ (\cite{MR1063344}, \cite{MR2282409}), and this is
the notion we adopt in this paper.

The CSM class of $X$ encodes interesting information about the variety $X$. For example,
if $X\subseteq \Pbb^n_\Cbb$ is a complex projective variety, then the degrees of the 
components of $\csm(X)$ carry the same information as the topological Euler characteristics
of its general linear sections
(\cite{MR3031565}). Chern-Schwartz-MacPherson classes of
classical varieties such as Schubert varieties and determinantal varieties have been studied 
extensively and are the object of current research (see e.g., \cite{rimava}, \cite{MR3594289},
\cite{MR3743175}, \cite{rimafe}, \cite{AMSS}).

\subsection{}
In~\cite{MR2001i:14009} we proved a formula for the Chern-Schwartz-MacPherson class
of a {\em hypersurface\/} in a nonsingular variety, in terms of the Segre class of its singularity 
subscheme. (In particular, this yields a formula for the topological Euler characteristic of
arbitrary hypersurfaces of nonsingular varieties.) Applications include computations in 
enumerative geometry (\cite{MR1641591}), 
singularities of logarithmic foliations (\cite{MR2197406}),
Sethi-Vafa-Witten-type formulas (\cite{MR2495684}), and others.
By inclusion-exclusion, the case of hypersurfaces suffices in
order to compute CSM classes of subschemes of nonsingular varieties. This fact is at the
root of most implemented algorithms for the computation of CSM classes in projective
spaces and more general varieties: see \cite{MR1956868}, \cite{MR3484270}, 
\cite{MR3385954}, \cite{MR3608162} and others. (To our knowledge, 
the algorithm presented in~\cite{MR3044490} is the {\em only one\/} currently available 
that does not rely on the result for hypersurfaces from~\cite{MR2001i:14009} and 
inclusion-exclusion.) 

One obvious problem with inclusion-exclusion is that the number of needed computations
grows exponentially with the number of defining hypersurfaces. For this, and for more 
conceptual reasons, it would be desirable to have direct generalizations of the 
result in~\cite{MR2001i:14009} to more general schemes. Such a 
generalization should express the CSM class of a scheme $X$ in terms of the Segre 
class of a related scheme determined by the singularities of $X$, without invoking
inclusion-exclusion. We raised this problem in~\cite[\S4.1]{MR2143071}, and to our
knowledge it has so far remained open in the intended generality. J.~Fullwood 
(\cite{MR3165594}) gave an answer for certain global complete 
intersections\footnote{Fullwood considers complete intersections $M_1\cap\cdots \cap M_k$ 
of hypersurfaces such that $M_1\cap \cdots \cap M_{k-1}$ is nonsingular.}; complete and local
complete intersections are also treated in references with different viewpoints
(among these \cite{MR1873009}, \cite{MR3053711}, \cite{miclalci}); but a result
along the lines envisioned above had to our knowledge not been formulated, even in 
the (unrestricted) complete intersection case.

The purpose of this article is to rectify this situation. For $\iota: X\hookrightarrow V$ 
an {\em arbitrary\/} closed embedding of a scheme $X$ in a nonsingular variety $V$, 
we will provide a formula for $\iota_*\csm(X)\in A_*V$ in terms of the Segre class of 
an associated subscheme of a projective bundle over $V$. In the 
hypersurface case, this formula will agree with the result of~\cite{MR2001i:14009}.
In the case of local complete intersections, it will yield an expression for (the 
push-forward to $V$ of) the so-called {\em Milnor class\/} of $X$. In general, the
formula will make no assumptions on $X$ other than it can be embedded as a closed
subscheme of a nonsingular variety. In fact, the formula will have nontrivial content 
even if $X$ is nonsingular.

\subsection{}\label{ss:projr}
In this introduction we present the result in the particular case in which $V=\Pbb^n$. This 
leads to some simplifications, and is possibly the most useful in concrete computations. 
In~\S\ref{s:stat} we will state the formula for arbitrary nonsingular ambient varieties~$V$.

Let $\iota: X\hookrightarrow \Pbb^n$ be a closed embedding. 
We may choose generators $F_0,\dots, F_r$ for any homogeneous ideal defining $X$ in 
$\Pbb^n$, such that $\deg F_i$ is an integer $d$ independent of $i$. 
Let $\cY$ denote the subscheme of $\Pbb^n \times \Pbb^r$ defined by the ideal
\[
\left( F_0,\dots, F_r\right) + \left(y_0 \frac{\partial F_0}{\partial x_i} + \cdots +
y_r \frac{\partial F_r}{\partial x_i}\right)_{i=0,\dots, n}\quad.
\]
Here $x_0,\dots, x_n$ are homogeneous coordinates in $\Pbb^n$, and $y_0,\dots,
y_r$ are homogeneous coordinates in $\Pbb^r$. Denote by $\pi:\Pbb^n\times\Pbb^r
\to \Pbb^n$ the projection and let $H$, resp., $h$ denote the pull-backs of the hyperplane 
classes from~$\Pbb^n$, resp., $\Pbb^r$. 

\begin{theorem}\label{thm:mainpr}
With notation as above, assume $r\ge n$. Then 
\begin{equation}\label{eq:mainprfor}
\iota_*\csm(X) = \pi_*\left(\frac{(1+H)^{n+1} (1+h)^{r+1}}{1+dH+h} 
\left(s(\cY, \Pbb^n\times \Pbb^r)^\vee \otimes_{\Pbb^n\times \Pbb^r}\cO(dH+h)\right)\right)\quad.
\end{equation}
\end{theorem}

\noindent (This statement uses the notation $\otimes$, ${}^\vee$ introduced 
in~\cite[\S2]{MR96d:14004}. We will recall this notation in~\S\ref{ss:tmt}.)
For instance, the degree of the class on the right-hand side equals the Euler characteristic
of $X$.

Note that the choices of the integer $d\gg 0$ and of the generators $F_i$ are arbitrary.
In particular, we could choose some of the $F_i$ to coincide, or even to be $0$, in order
to guarantee that $r\ge n$. Every such choice leads to an expression for the CSM class
of $X$.

The main result we will present in~\S\ref{s:stat} (Theorem~\ref{thm:main}) will pose no 
restriction on the number~$r$ 
of generators of a defining ideal for $X$. The case $r\ge n$ leads to a direct formula for 
the CSM class of $X$, of which Theorem~\ref{thm:mainpr} is a particular case. Another 
case of interest is $r+1=\codim X$, i.e., the case of a global complete intersection. Recall 
that the {\em Milnor class\/} of a complete intersection $X$ is the (signed) difference of 
its CSM class and of the Chern class of the virtual tangent bundle of $X$:
\begin{equation}\label{eq:milnordef}
\cM(X) = (-1)^{\dim X}\left(\cvir(X)-\csm(X)\right)\quad.
\end{equation}
(See~e.g., \cite{MR2002g:14005}. To our knowledge, this terminology is due to 
S.~Yokura, \cite{MR1695362, MR1720876}.)

\begin{theorem}\label{thm:mainprci}
Let $\iota: X\hookrightarrow \Pbb^n$ be a complete intersection of $r+1$ hypersurfaces 
of degree~$d$. Then with notation as above{\small
\begin{equation}\label{eq:mainprforci}
(-1)^{\dim X+1}\iota_*\cM(X) = \pi_*\left(\frac{(1+H)^{n+1} (1+h)^{r+1}}{1+dH+h}
\left(s(\cY, \Pbb^n\times \Pbb^r)^\vee \otimes_{\Pbb^n\times \Pbb^r}\cO(dH+h)\right)\right)\quad.
\end{equation}}
\end{theorem}

It is worth stressing that the right-hand sides in~\eqref{eq:mainprfor} 
and~\ref{eq:mainprforci} are {\em identical.\/} The claim is that for $r\gg 0$ this formula yields 
the CSM class of $X$, while if $X$ is a complete intersection of $r+1$ hypersurfaces of a fixed 
degree $d$, the same formula yields the Milnor class of $X$ (up to a sign).

\subsection{}\label{ss:exa}
The formulas stated above can be implemented easily in Macaulay2 (\cite{M2}),
using the Segre class function in the package {\tt CharacteristicClasses.m2} by M.~Helmer
and C.~Jost (\cite{HJ}): this package can handle Segre classes of subschemes of products 
of projective spaces, and computing the push-forward amounts to simply extracting the 
coefficient of $h^r$.
The same package also implements the computation of CSM classes (by the 
inclusion-exclusion method mentioned above), so it may be used as an independent verification
of results obtained applying Theorems~\ref{thm:mainpr} and~\ref{thm:mainprci}.
We illustrate the application of Theorem~\ref{thm:mainprci} to the complete intersection
of the singular hypersurfaces 
\[
Z_1: \{x_1x_2x_3=0\}\quad,\quad Z_2: \{x_0 x_1^2 +x_2^3=0\}
\]
in $\Pbb^6$. The scheme $X=Z_1\cap Z_2$ consists of three components of codimension~$2$,
one of which (supported on a linear subspace) is nonreduced. The following Macaulay2 session 
implements the computation of the Segre class $s(\cY,\Pbb^6\times \Pbb^1)$ for the scheme 
$\cY$ associated with $X$. (We omit inessential output.)

{\tiny\begin{verbatim}
i1 : load ("CharacteristicClasses.m2");

i2 : R=MultiProjCoordRing({6,1});

i3 : r= gens R

o3 = {x , x , x , x , x , x , x , x , x }
       0   1   2   3   4   5   6   7   8

i4 : Y=ideal(r_1*r_2*r_3,r_1^2*r_0+r_2^3,r_8*r_1^2,r_7*r_2*r_3+r_8*2*r_0*r_1,r_7*r_1*r_3+r_8*3*r_2^2,r_7*r_1*r_2)

                       2    3   2                                 2
o4 = ideal (x x x , x x  + x , x x , x x x  + 2x x x , x x x  + 3x x , x x x )
             1 2 3   0 1    2   1 8   2 3 7     0 1 8   1 3 7     2 8   1 2 7

i5 : Segre(Y)

         6         6       5        5      4        4      3      3     2      2
o5 = 181h h  - 240h  - 167h h  + 72h  + 69h h  - 16h  - 19h h  + h  + 3h h  + h
         1 2       1       1 2      1      1 2      1      1 2    1     1 2    1
\end{verbatim}}

\noindent
Here $H={\tt h_1}$ and $h={\tt h_2}$. The result is that
\begin{multline*}
s(Y,\Pbb^6\times \Pbb^1) = (H^2 + 3H^2 h + H^3 -19 H^3 h -16 H^4 \\
+69 H^4 h + 72 H^5-167 H^5 h -240 H^6 + 181 H^6 h)\cap [\Pbb^6\times\Pbb^1]
\end{multline*}
(after push-forward to the ambient space). It is then straightforward to compute
\begin{multline*}
\frac{(1+H)^7 (1+h)^2}{1+3H+h}
\left(s(\cY, \Pbb^6\times \Pbb^1)^\vee \otimes_{\Pbb^6\times \Pbb^1}\cO(3H+h)\right) \\
=(H^2-3H^3+H^4-17H^5+42H^6) - (4H^2-9H^3+29H^4-107H^5+363H^6)h\quad.
\end{multline*}
According to Theorem~\ref{thm:mainprci},
\[
\iota_*\cM(X) = \left(4H^2-9H^3+29H^4-107H^5+363H^6\right)\cap [\Pbb^6]
\]
is the Milnor class of $X$. Since $X$ is a complete intersection of two
hypersurfaces of degree~$3$,
\[
\iota_* \cvir(X) =\frac{(1+H)^7}{(1+3H)^2}\cap [\Pbb^6] = 
\left(9H^2+9H^3+54H^4-90H^5+369H^6\right)\cap \Pbb^6\quad.
\]                                                                                                                                              
It follows that
\[
\iota_*\csm(X) = \iota_*(\cvir(X)-(-1)^{\dim X}\cM(X)) 
= (5H^2+18H^3+25H^4+17H^5+6H^6)\cap [\Pbb^6]\quad.
\]
(This can be confirmed independently by~\cite{HJ}.)

Note that $Z_1$ and $Z_2$ are both singular, and their singular loci have nonempty
intersection. It follows that in this example the complete intersection $X$ cannot be
represented as a hypersurface in a nonsingular subvariety of $\Pbb^6$; therefore
it does not satisfy the hypotheses of~\cite{MR3165594} and~\cite{MR3659416}.

The requirement in Theorem~\ref{thm:mainprci} that the degrees of the defining 
hypersurfaces coincide leads to the particularly explicit formula~\eqref{eq:mainprforci}.
The more general result presented in~\S\ref{s:stat} (Corollary~\ref{cor:main}) will
dispense of this requirement; an expression for the Milnor class will be obtained
for every local complete intersection represented as the zero-scheme of a regular
section of a vector bundle on a nonsingular variety.

\subsection{}
The paper is organized as follows. In \S\ref{s:stat} we provide a full statement of the main
result (Theorem~\ref{thm:main}) and give several illustrating examples, including the 
derivation of Theorem~\ref{thm:mainpr} and Theorem~\ref{thm:mainprci}.
In \S\ref{s:proof} we prove the main result.
The proof relies on the hypersurface case given in~\cite{MR2001i:14009}, on 
calculus of constructible functions, and on intersection-theoretic computations. 
A key ingredient in the proof is the construction of an auxiliary hypersurface, an idea we 
borrow from~\cite{miclalci}. As we show in~\S\ref{ss:alter}, the scheme~$\cY$ is the 
singularity subscheme of this hypersurface. In~\cite{miclalci} this hypersurface is 
constructed in the local complete intersection case, and it is also used to obtain formulas 
for Milnor classes (see~\S\ref{ss:compa}). 
We note here that the scheme $\cY$ was also considered by T.~Ohmoto (\cite{Ohm})
and X.~Liao (\cite{Liao}).

As an application of the main result, we expand in \S\ref{s:zeta} on the case of subschemes 
of projective space. Recent results on Segre classes lead to alternative, and in some way
more efficient, formulations of the result in this case.

\subsection{}
The result implies that expressions such as~\eqref{eq:mainprfor} (or the more general
version~\eqref{eq:mainfor} given in~\S\ref{s:stat}) are independent of the choices:
in the case of subschemes $X$ of $\Pbb^n$ these choices are the degree $d\gg 0$ of 
the generators of a defining ideal, the number $r\ge n$ of generators, or in fact the 
generators themselves. We do not know a more direct 
proof of this independence. In fact, $X$ may be replaced by {\em any\/} scheme with the 
same support as $X$ without affecting these expressions. While this fact is an immediate
consequence of the main result, it seems quite nontrivial in itself. 

We illustrate this fact with an example. Let $X$
be the scheme with ideal $(x^2,xy)$ in $\Pbb^2$; so $X$ is supported on a line $\Pbb^1$,
with an embedded component at the point $x=y=0$. We choose the generators $F_0=x^2$,
$F_1=xy$, $F_2=0$ for the ideal of $X$, which determine as described above the subscheme
$\cY$ of $\Pbb^2\times \Pbb^2$ with ideal
\[
(x^2,xy, 2y_0x+y_1y,y_1x)\quad.
\]
According to {\tt CharacteristicClass.m2},
\[
s(\cY,\Pbb^2\times \Pbb^2) = (Hh+H^2-Hh^2-2H^2h+3H^2h^2)\cap [\Pbb^2\times \Pbb^2]\quad;
\]
it follows that the class appearing on the right-hand side of~\eqref{eq:mainprfor} is
\[
(H^2+(H-H^2)h+(H+2H^2)h^2)\cap [\Pbb^2\times \Pbb^2]\quad.
\]
Performing the same computation using the reduced $\Pbb^1$, with generators 
$F_0=x,F_1=0,F_2=0$, yields the class
\[
((H+H^2)h+(H+2H^2)h^2)\cap [\Pbb^2\times \Pbb^2]\quad.
\]
The classes differ, but the coefficient of $h^2$, i.e., their push-forward to $\Pbb^2$, agree 
(and yield $\csm(\Pbb^1)=c(T\Pbb^1)\cap [\Pbb^1]$ as prescribed by
Theorem~\ref{thm:mainpr}). 

The results of this paper will prove that CSM classes of schemes $X$ with the same support 
agree as classes in the Chow group of every nonsingular variety containing $X$. It would be 
desirable to have a direct proof that classes obtained in this fashion are independent of all 
choices as classes in the Chow group $A_*X_\text{red}$.

\subsection{}
The formula for hypersurfaces in~\cite{MR2001i:14009} may be seen as a manifestation of 
an identity of characteristic cycles; see~\cite{MR2002g:14005} for this point of view and an 
alternative proof of the formula in~\cite{MR2001i:14009}. It is a natural project to provide a 
characteristic cycle version of the generalization obtained in this note.

To our knowledge, the hypersurface formula is not implied by the very general motivic
theory for characteristic classes of hypersurfaces, as presented in~\cite{MR2680178} 
(and extended to complete intersections in~\cite{MR3053711}). {\em A fortiori,\/} a
direct relation between the generalization presented here and the theory of motivic
Hirzebruch classes would be surprising and very interesting. Equally interesting would
be a connection with Yokura's `motivic Milnor classes' (\cite{MR2671765}).

Finally, we note that Fullwood and D.~Wang have proposed a conjectural generalization of the 
hypersurface formula (\cite{FullwoodWang}), in terms of a blow-up construction; they 
prove that this formulation is correct for certain complete intersections. It is straightforward
to express our results in this note in terms of the blow-up along the scheme $\cY$, as
this blow-up may be used to compute the Segre class of $\cY$. It would be
interesting to relate the center of the blow-up in~\cite{FullwoodWang} to $\cY$.
\smallskip

\subsection{Acknowledgments}
This work was supported in part by NSA grant H98230-16-1-0016.
This work was carried out while the author was visiting the University of Toronto.
The author thanks the University of Toronto for the hospitality.


\section{Statement}\label{s:stat}

\subsection{Preliminaries}\label{ss:CSM}
We work over an algebraically closed field $k$ of characteristic~$0$. (This requirement is
needed for Chern-Schwartz-MacPherson classes.) Throughout the paper, $X$ will denote
a $k$-scheme which can be embedded as a closed subscheme of a nonsingular variety~$V$.

The {\em Chern-Schwartz-MacPherson\/} (CSM) class of $X$ may be defined as an element
in the Chow group $A_*X$ of $X$. It is determined by the requirement that if $X$ is nonsingular,
then $\csm(X)=c(TX)\cap [X]$ and by a specific behavior with respect to proper morphisms,
which we now recall. 

We can associate with each $X$ the group of {\em constructible functions\/} $F(X)$, i.e., 
integer-valued functions on $X$ which may be obtained as finite linear combinations of
indicator functions on subvarieties of $X$: $\varphi=\sum_W \one_W$, where the sum
ranges over finitely many closed subvarieties $W$ of $X$, and $\one_W(p)=1$ if $p\in W$,
$\one_W(p)=0$ if $p\not\in W$. The assignment $X\leadsto F(X)$ defines a {\em covariant
functor\/} to the category of abelian groups, if we prescribe the following push-forward
for proper maps: if $f:X\to Y$ is a proper morphism, a homomorphism
$f_*: F(X)\to F(Y)$ is defined by requiring $f_*(\one_W)$ to be the function
\[
f_*(\one_W)(p)=\chi(f^{-1}(p)\cap W)\quad,
\]
where $\chi$ denotes the topological Euler characteristic if $k=\Cbb$, and a suitable analogue
over more general fields (see e.g., \cite[\S2.1]{MR3031565}).

According to a theorem of MacPherson (\cite{MR0361141}) and extensions of this result 
to the context used here, there exists a natural transformation from 
$F$ to the Chow group functor~$A_*$, such that the indicator function $\one_X$ is sent to 
$c(TX)\cap [X]$ if $X$ is nonsingular. The class $\csm(X)$ is the image of $\one_X$ in $A_*X$, 
regardless of the singularities of $X$. More generally, we denote by $\csm(\varphi)$ the image 
of $\varphi\in F(X)$ in $A_*X$. With this notation, if $f:X\to Y$ is a proper map, then
\[
\csm(f_*\varphi)=f_* \csm(\varphi)\quad.
\]
This covariance property implies easily that the natural transformation is unique: indeed, by 
resolution of singularities the CSM class of any scheme $X$ as above is determined by the 
CSM classes of a suitable selection of nonsingular varieties mapping to $X$.
Also note that if the Euler characteristic of the fibers of a proper morphism $f:X\to Y$ is
a constant $\chi$, then covariance implies that
\begin{equation}\label{eq:cova}
f_* \csm(X) = \chi\cdot \csm(Y)\quad.
\end{equation}

By abuse of language, if $X\subseteq V$, then we may denote by $\csm(X)$ the class 
$\csm(\one_X)$ in the Chow group {\em of $V$.\/} 
With this convention, the CSM class 
satisfies a basic inclusion-exclusion principle: for $X,Y\subseteq V$, we have
\[
\csm(X\cup Y) = \csm(X)+\csm(Y)-\csm(X\cap Y)\quad.
\]
This is often useful in concrete computations.

\subsection{The scheme $\cY$}\label{ss:TsY}
We now fix a nonsingular variety $V$, and a closed subscheme $X\subseteq V$.
We denote by $\iota$ the inclusion map $X\hookrightarrow V$.

We may view $X$ as the zero-scheme of a section of a vector bundle $E$. Indeed, we may
choose $E=\Spec(\Sym \cE)$, where $\cE$ is any locally free sheaf surjecting onto the
ideal sheaf $\cI_{X,V}$ of $X$ in $V$; the composition $s^\vee: \cE\twoheadrightarrow \cI_{X,V}
\hookrightarrow \cO_V$ corresponds to a section $s: V\to E$, such that $X=Z(s)$.
(Cf.~\cite[B.8.2]{85k:14004}.)

Recall that we have an exact sequence
\[
\xymatrix{
{\cI_{X,V}/\cI_{X,V}^2} \ar[r] & \Omega_V|_X \ar[r] & \Omega_X \ar[r] & 0
}
\]
(\cite[Proposition~8.12]{MR0463157}).
Restricting the surjection $\cE \twoheadrightarrow \cI_{X,V}$ to $X$ and composing with
the first morphism in this sequence determines a morphism of locally free sheaves
\[
\cE|_X \longrightarrow \Omega_V|_X\quad,
\]
or equivalently a morphism of vector bundles on $X$:
\begin{equation}\label{eq:morphEO}
\phi:\quad E^\vee|_X \longrightarrow T^*V|_X\quad.
\end{equation}
Now consider the projective bundle (of lines) $\rho: \Pbb(E^\vee|_X) \to X$.
Composing the pull-back of~\eqref{eq:morphEO} with the inclusion of the tautological 
subbundle, we obtain a morphism
\[
\sigma_{\cY}:\quad \cO(-1) \to \rho^*E^\vee|_X \to \rho^*T^*V|_X \quad.
\]
of vector bundles over $\Pbb(E^\vee|_X)$.

\begin{defin}
With notation as above, we define $\cY\subseteq \Pbb(E^\vee|_X)$ to be the zero-scheme
of~$\sigma_{\cY}$: $\cY=Z(\sigma_\cY)$.
\qede\end{defin}

Set-theoretically, $\cY$ consists of
points $(\underline e,x)$, with $\underline e$ in the fiber of $\Pbb(E^\vee|_X)$ at $x\in X$,
such that $\underline e\in \ker\phi_x$, where $\phi$ is the morphism in~\eqref{eq:morphEO}.

In local analytic coordinates $(x_1,\dots, x_n)$
for $V$ at $x$, $s^\vee$ describes the ideal of $X$ in terms of a choice of generators 
$f_0,\dots, f_r\in k[[x_1,\dots, x_n]]$, where $r+1=\rk E$.
The morphism $\phi$ is given by the $n\times (r+1)$ matrix
\[
\begin{pmatrix}
\frac{\partial f_0}{\partial x_1} & \cdots & \frac{\partial f_r}{\partial x_1} \\
\vdots & \ddots & \vdots \\
\frac{\partial f_0}{\partial x_n} & \cdots & \frac{\partial f_r}{\partial x_n}
\end{pmatrix}
\]
and $\cY$ is defined as a subscheme of $\Pbb(E^\vee|_X)$ by the vanishing
\begin{equation}\label{eq:basvan}
\begin{pmatrix}
\frac{\partial f_0}{\partial x_1} & \cdots & \frac{\partial f_r}{\partial x_1} \\
\vdots & \ddots & \vdots \\
\frac{\partial f_0}{\partial x_n} & \cdots & \frac{\partial f_r}{\partial x_n}
\end{pmatrix}\cdot
\begin{pmatrix}
e_0 \\ \vdots \\ e_r
\end{pmatrix} = 0
\end{equation}
with $\underline e = (e_0:\cdots :e_r)$.
Thus, $\cY$ detects linear relations among differentials of the chosen generators for $X$.

\begin{remark}
As Terry Gaffney pointed out, $\cY$ may therefore be viewed as a `Tyurina transform' 
associated with the morphism~$\phi$.
\qede\end{remark}

One source of such relations are the singularities of $X$. 

\begin{example}\label{ex:hype}
Assume that $r=0$, so that $X$ is the hypersurface in $V$ with local equation $f_0=0$. 
Then $\cY\subseteq X\times \Pbb^0 \cong X$ is locally defined by the vanishing of the 
partials of $f_0$; that is, in this case $\cY$ is the singularity subscheme of $X$. 
\qede\end{example}

In the case of Example~\ref{ex:hype}, $\cY$ is empty if $X$ is nonsingular. More generally,
$\cY$ is empty if $X$ is a smooth complete intersection realized as the zero-scheme of
a regular section of a bundle $E$ of rank equal to $\codim X$; in this case, $\phi_x$ has full 
rank for all $x\in X$.

However, this is not typical. Linear relations among differentials of the generators may be 
due to reasons other than the singularities of $X$. For example, two generators may coincide
or one of the generators may be identically $0$.

\begin{example}\label{ex:hypeag}
Let $X\subseteq V$ be a smooth hypersurface, given as the zero-scheme of a section~$f$ of
a line bundle $L$. With notation as above, let $E=L^{\oplus r+1}$, with $r>0$, and let 
$s=(f,\dots, f)$. Then $\cY$ is a $\Pbb^{r-1}$ bundle over $X$.
\qede\end{example}

\subsection{The main theorem}\label{ss:tmt}
Let $\cV=\Pbb(E^\vee)$, and let $\pi: \cV\to V$ denote the projection:
\[
\xymatrix{
\Pbb(E^\vee|_X) \ar[d]_{\rho} \ar@{^(->}[r] & \cV \ar[d]^\pi \\
X \ar@{^(->}[r] & V
}
\]
With this notation, $\Pbb(E^\vee|_X)=\pi^{-1}(X)$.

It will be useful to view $\cY$ as a subscheme of $\cV$; as such, the ideal of $\cY$ is 
generated by the pull-back of $\cI_{X,V}$ and by the relations~\eqref{eq:basvan}. 
The closed embedding $\cY\subseteq \cV$ determines the {\em Segre class\/}
$s(\cY,\cV)\in A_*\cY$ (\cite[Chapter~4]{85k:14004}). We implicitly often view this class
as a class in $A_*\cV$, omitting the evident push-forward notation.

We will need the following notation from~\cite[\S2]{MR96d:14004}. Let $M$ be an
ambient variety, and let $Z$ be a subscheme of $M$. Further, let $\cL$ be a line bundle
on $Z$. For $\alpha\in A_*Z$, write $\alpha = \sum_i \alpha^{(i)}$, where $\alpha^{(i)}$
is the component of $\alpha$ with codimension $i$ {\em in $M$.\/} We define
\[
\alpha\otimes_M \cL := \sum_i c(\cL)^{-i}\cap \alpha^{(i)} \quad, \quad
\alpha^\vee := \sum_i (-1)^i \alpha^{(i)}\quad.
\]
The subscript $M$ may be omitted in context (and the notation ${}^\vee$ must be
understood in context, since it also depends on the dimension of the ambient variety $M$). 
This notation satisfies simple compatibility properties with the notion of dual of vector
bundles and of tensor product of vector bundles by line bundles, in terms of their effect
on Chern classes. Further, it is an {\em action\/} in the sense that if $\cL_1$ and $\cL_2$ 
are line bundles on $Z$, then $\alpha\otimes (\cL_1\otimes \cL_2)
=(\alpha\otimes \cL_1)\otimes \cL_2$. (See~\cite[Propositions~1 and~2]{MR96d:14004}.)

The following is our main result.

\begin{theorem}\label{thm:main}
Let $V$ be a nonsingular variety, and let $\iota: X\to V$ be a closed subscheme. 
Assume $X=Z(s)$ for a section $s$ of a vector bundle $E$ on $V$, and construct
$\cY, \cV$ as above. Then{\small
\begin{equation}\label{eq:mainfor}
\iota_* \csm(X)-\frac{c(TV)}{c(E)}\ctop(E)\cap [V] =
c(TV)\cap \pi_* \left(
\frac{c(\pi^* E^\vee\otimes \cO(1))}{c(\cO(1))} \cap
\left(s(\cY,\cV)^\vee\otimes_{\cV} \cO(1)\right)
\right)\quad.
\end{equation}}
\end{theorem}

The proof of Theorem~\ref{thm:main} is given in~\S\ref{s:proof}.
We record here the following consequence and several special cases illustrating
the statement.

\begin{corol}\label{cor:main}
With notation as above, let $\rk E=r+1$. Then
\begin{itemize}
\item If $r\ge \dim V$, then
\[
\iota_* \csm(X)= c(TV)\cap \pi_* \left(
\frac{c(\pi^* E^\vee\otimes \cO(1))}{c(\cO(1))} \cap
\left(s(\cY,\cV)^\vee\otimes_{\cV} \cO(1)\right)
\right)\quad.
\]
\item If $X$ is a local complete intersection in $V$ and $r+1=\codim X$, then
\[
\iota_* \cM(X)= (-1)^{\dim X+1} c(TV)\cap \pi_* \left(
\frac{c(\pi^* E^\vee\otimes \cO(1))}{c(\cO(1))} \cap
\left(s(\cY,\cV)^\vee\otimes_{\cV} \cO(1)\right)
\right)\quad.
\]
where $\cM(X)$ denotes the Milnor class of $X$.
\end{itemize}
\end{corol}

\begin{proof}
If $r\ge \dim V$, then $\rk E>\dim V$, hence $\ctop(E)=0$ for dimensional reasons.
The first formula follows then immediately from Theorem~\ref{thm:main}.

Concerning the second formula: If $X$ is a local complete intersection, $X$ is the zero
scheme of a section of a vector bundle $E$, and $\rk E=\codim X$, then $E|_X\cong
N_XV$, and $(TV|_X)/(E|_X)$ is the virtual tangent bundle of $X$. Further, $\ctop(E)\cap [V]
=\iota_*[X]\in A_*V$. Therefore
\[
\frac{c(TV)}{c(E)}\ctop(E)\cap [V] = \iota_* \cvir(X)
\]
in this case. By definition of Milnor class \eqref{eq:milnordef}, we have
\[
\iota_* \csm(X)-\frac{c(TV)}{c(E)}\ctop(E)\cap [V] = \iota_*(\csm(X)-\cvir(X)) = 
(-1)^{\dim X+1}\iota_* \cM(X)\quad,
\]
and the second formula follows from Theorem~\ref{thm:main}.
\end{proof}

\begin{example}
Let $V=\Pbb^n$. Every $X\subseteq V$ may be defined by a homogeneous ideal generated by 
forms of degree $d$, if $d\gg 0$. Choose such a $d$, and choose generators $F_0,\dots, F_r$
of~$H^0(\Pbb^n, \cI_{X,\Pbb^n}(d))$. View $(F_0,\dots, F_r)$ as a section of 
$E=\cO(d)^{\oplus(r+1)}$. We have 
\[
\Pbb(E^\vee) = \Pbb(\cO(-dH)^{\oplus(r+1)}) \cong \Pbb^n\times \Pbb^r\quad,
\]
where $H$ denotes the hyperplane class in $\Pbb^n$ (and its pull-back). Denoting by $h$ 
the hyperplane class in $\Pbb^r$, we have $c_1(\cO_{\Pbb(E^\vee)}(1))=dH+h$. Therefore
\begin{multline*}
c(TV)\cap \pi_*\left(\frac{c(\pi^* E^\vee\otimes \cO(1))}{c(\cO(1))} \cap
\left(s(\cY,\cV)^\vee\otimes_{\cV} \cO(1)\right)\right) \\
=(1+H)^{n+1}\cap \pi_*\left(\frac{(1-dH+(dH+h))^{r+1}}{1+dH+h} \cap
\left(s(\cY,\cV)^\vee\otimes_{\cV} \cO(dH+h)\right)\right)\quad,
\end{multline*}
and the formulas in Corollary~\ref{cor:main} specialize to Theorems~\ref{thm:mainpr} 
and~\ref{thm:mainprci}.
\qede\end{example}

\begin{example}
Let $X$ be a hypersurface in $V$, given as the zero-scheme of a section $f$ of the line
bundle $E=\cO(X)$. As noted in~Example~\ref{ex:hype}, $\cY$ equals the singularity
subscheme $JX$ of $X$ in this case. We have $\cV=\Pbb(E^\vee)=\Pbb(\cO(-X))\cong V$,
and $\pi$ is the identity under this identification. The line bundle $\cO(1)$ is tautologically 
isomorphic to $\cO(X)$. Since $\rk E=\codim X$, the second formula in Corollary~\ref{cor:main} 
applies, giving
\[
\iota_* \cM(X)= (-1)^{\dim X} c(TV)\cap \left(
\frac{c(\cO(-X)\otimes \cO(X))}{c(\cO(X))} \cap
\left(s(JX,V)^\vee\otimes_V \cO(X)\right)
\right)\quad,
\]
that is,
\[
\iota_*(\csm(X)-\cvir(X)) = c(TV)\cap \left( c(\cO(X))^{-1} \cap
\left(s(JX,V)^\vee\otimes_V \cO(X)\right)
\right)\quad.
\]
This is the main result of~\cite{MR2001i:14009} (after push-forward by $\iota_*$ to the 
ambient nonsingular variety $V$).
\qede\end{example}

\begin{example}
To illustrate the dependence of the result on the rank of $E$ in a particularly transparent
case, let $X\subseteq V$ be a smooth hypersurface. We can view $X$ as the zero scheme
of a section of $E=\cO(X)^{\oplus(r+1)}$, of the form (for example) $(f,\dots,f)$. As observed
in Example~\ref{ex:hypeag}, $\cY\subseteq \Pbb(E^\vee|_X)$ is then a $\Pbb^{r-1}$ bundle
over $X$. We can identify $\cV=\Pbb(E^\vee)$ with $V\times \Pbb^r$; let $h$ be the pull-back
of the hyperplane class from the second factor, and $\pi$ the projection onto the first factor.
Then $\cO_{\Pbb(E^\vee)}(1)\cong \cO(h+\pi^*X)$, and $\cY$ is a complete intersection
of $\Pbb(E^\vee|_X)=\pi^{-1}(X)$ and a hyperplane in the second factor (with equation
$e_0+\cdots+e_r=0$). We have
\[
s(\cY,\cV) = \frac{h\cdot \pi^*X}{(1+h)(1+\pi^*X)}\cap [\cV]
\]
as a class in $A_*\cV$, hence
\[
s(\cY,\cV)^\vee\otimes_\cV \cO(1) = \frac{h\cdot \pi^*X}{(1-h)(1-\pi^*X)}\otimes_\cV \cO(h+\pi^*X)
\cap [\cV]= \frac{h\cdot \pi^*X}{(1+\pi^*X)(1+h)}\cap [\cV]\quad.
\]
(Using~\cite[Proposition~1]{MR96d:14004}.)
Therefore, omitting evident pull-backs,
\begin{multline*}
\frac{c(E^\vee\otimes_\cV \cO(1))} {c(\cO(1))}\cap \left(s(\cY,\cV)^\vee\otimes_\cV \cO(1) \right) 
=\frac{(1+h)^{r+1}}{1+h+X}\cdot \frac{h\cdot X}{(1+X)(1+h)}\cap [\cV] \\
=\frac{(1+h)^r}{(1+h+X)(1+X)}\cap [\cV]
=(1+h)^r\left(\frac{X}{1+X} - \frac{X}{1+h+X}\right)\cap [\cV]\quad.
\end{multline*}
The push-forward of this class to $\cV$ is determined by 
the coefficient of $h^r$ in this expression. It is easy to verify that this equals
\[
\left(\frac{X}{1+X}-\frac{X^{r+1}}{(1+X)^{r+1}}\right)\cap [V]
\]
and it follows that
\begin{multline*}
c(TV)\cap \pi_*\left(\frac{c(\pi^*E^\vee\otimes_\cV \cO(1))} {c(\cO(1))}\cap 
\left(s(\cY,\cV)^\vee\otimes_\cV \cO(1) \right) \right) \\
=c(TV)\cap \left(\frac{X}{1+X}-\frac{X^{r+1}}{(1+X)^{r+1}}\right)\cap [V]
=c(TX)\cap [X]-c(TV)\frac{c_1(\cO(X))^{r+1}}{c(\cO(X))^{r+1}}\cap [V]
\end{multline*}
in agreement with Theorem~\ref{thm:main}.
\qede\end{example}

\begin{example}\label{ex:caseV}
If $X=V$, we can represent $X$ as the zero-scheme of the zero-section of any vector bundle~$E$
on $V$. In this case $\cY=\cV$, so that
\[
s(\cY,\cV)^\vee \otimes_\cV \cO(1) = [\cV]^\vee \otimes \cV \cO(1) = [\cV]\quad.
\]
Theorem~\ref{thm:main} reduces then to the statement that
\begin{equation}\label{eq:seq0}
\pi_*\left(\frac{c(\pi^*E^\vee\otimes \cO(1))}{c(\cO(1))}\cap[\cV]\right)
=\left(1-\frac{\ctop(E)}{c(E)}\right)\cap [V] \quad.
\end{equation}
This statement will in fact be an ingredient in the proof of Theorem~\ref{thm:main}, and will
be (independently) proven in~\S\ref{ss:cFX}.
\qede\end{example}


\section{Proof}\label{s:proof}

\subsection{}
The proof of Theorem~\ref{thm:main} relies on several ingredients.
In~\S\ref{ss:alter} we give an alternative description of the scheme~$\cY$ defined in
\S\ref{ss:TsY}, as the singularity subscheme of a hypersurface~$\cX$ in $\Pbb(E^\vee)$. 
In~\S\ref{ss:cFX} we compute the push-forward of $\cvir(\cX)$, by standard techiques
in intersection theory. In~\S\ref{ss:calc} we compute the push-forward of~$\csm(\cX)$ by 
applying the functoriality of CSM classes, and in~\S\ref{ss:fine} we use the main result 
of~\cite{MR2001i:14009} to establish~Theorem~\ref{thm:main}.

In~\S\ref{ss:compa} we comment on related work of Callejas-Bedregal, Morgado, and Seade
concerning Milnor classes of local complete intersections (\cite{miclalci}). The 
hypersurface~$\cX$ we use in the proof of~Theorem~\ref{thm:main} was to our knowledge
first introduced in~\cite{miclalci} (in the local complete intersection case).

\subsection{Alternative description of $\cY$}\label{ss:alter}
With notation as in~\S\ref{ss:TsY}, dualize the inclusion of the tautological subbundle
$\cO(-1)\to \pi^*(E^\vee)$ to obtain a canonical morphism $\epsilon: \pi^*(E) \to \cO(1)$.
Composing with the pull-back of $s$ gives a section of $\cO(1)$ on $\cV$:
\begin{equation}\label{eq:scX}
\xymatrix{
\sigma_\cX:\quad \cV \ar[r]^-{\pi^* s} & \pi^*(E) \ar[r]^\epsilon & \cO(1)\quad.
}
\end{equation}
We let $\cX$ denote the hypersurface of $\cV$ defined as the zero-scheme of
$\sigma_\cX=\epsilon\circ \pi^*s$.

\begin{lemma}\label{lem:singy}
The scheme $\cY$ is the singularity subscheme of $\cX$.
\end{lemma}

\begin{proof}
In local analytic coordinates as above, $\cX$ is given by the equation
\begin{equation}\label{eq:eqcX}
y_0 f_0 +\cdots + y_rf_r=0\quad,
\end{equation}
whose Jacobian ideal defines the singularity subscheme of $\cX$. From this and the coordinate
description of $\cY$ given in~\S\ref{ss:TsY}, the statement is clear. More intrinsically, the ideal 
of $\cX$ in~$\cV$ is $\cO(-1)$; hence we have a canonical morphism
\[
\xymatrix{
\cO(-1)|_\cX \ar[r] & \Omega_{\cV}|_{\cX}\quad,
}
\]
and equivalently (tensor by $\cO(1)$) a section
\[
\xymatrix{
\sigma_{J\cX}:\quad \cX \ar[r] &  T^*\cV|_\cX\otimes \cO(1)
}
\]
of a twist of the cotangent bundle to $\cV$.
By definition, the singularity subscheme of $\cX$ is the zero-scheme of this section.
Now, we have the dual Euler exact sequence
\[
\xymatrix{
0 \ar[r] & (T^*_{\cV/V})\otimes \cO(1) \ar[r] &
\pi^*E \ar[r]^\epsilon & \cO(1) \ar[r] & 0 \\
& & \cV \ar[u]^{\pi^* s} \ar[ur]_{\sigma_\cX}
}
\]
where $T^*_{\cV/V}$ is the relative cotangent bundle. As $\cX$ is the zero-scheme of $\sigma_\cX$, 
we obtain a section
\[
\sigma':\quad \cX \to (T^*_{\cV/V})\otimes \cO(1)|_\cX
\]
which is seen to be compatible with $\sigma_{J\cX}$: the diagram
\[
\xymatrix{
0 \ar[r] & \pi^* T^*V\otimes \cO(1)|_\cX \ar[r] & T^*\cV\otimes \cO(1)|_\cX \ar[r] &
(T^*_{\cV/V})\otimes \cO(1)|_\cX \ar[r] & 0 \\
& & \cX \ar[u]^{\sigma_{J\cX}} \ar[ur]_{\sigma'}
}
\]
commutes. The singularity subscheme of $\cX$, i.e., the zero-scheme $Z(\sigma_{J\cX})$, 
is contained in $Z(\sigma')=Z(\pi^*s)=\pi^{-1}(X)=\Pbb(E^\vee|_X)$. 
Restricting to $\Pbb(E^\vee|_X)$, $\sigma_{J\cX}$ induces a section
\[
\xymatrix{
\sigma'':\quad \Pbb(E^\vee|_X) \ar[r] & (\pi^* T^*V\otimes \cO(1))|_{\Pbb(E^\vee|_X)}
=\rho^* T^*V|_X\otimes \cO(1)\quad,
}
\]
such that the singularity subscheme of $\cX$ equals $Z(\sigma'')$. It is now easy
to check that $\sigma''$ agrees with $\sigma_\cY\otimes \cO(1)$, and it follows that
the singularity subscheme of $\cX$ coincides with $Z(\sigma_{\cY})=\cY$.
\end{proof}

\subsection{}\label{ss:cFX}
Theorem~\ref{thm:main} will follow from the computation of push-forwards of characteristic 
classes of~$\cX$. In this subsection we compute $\pi_*(\cvir(X))$. For this purpose it will be
useful to prove identity~\eqref{eq:seq0}: as noted in~\S\ref{ex:caseV}, this simple statement is
a particular case of Theorem~\ref{thm:main}, and it turns out that it is in fact one of the 
ingredients in its proof.

\begin{lemma}\label{lem:pfE1}
Let $E$ be a vector bundle on a variety $V$, and let $\pi: \Pbb(E^\vee)\to V$ be the projective 
bundle (of lines) of its dual $E^\vee$. Then
\[
\pi_*\left(\frac{c(\pi^*E^\vee\otimes \cO(1))}{c(\cO(1))}\cap[\Pbb(E^\vee)]\right)
=\left(1-\frac{\ctop(E)}{c(E)}\right)\cap [V] \quad.
\]
\end{lemma}

\begin{proof}
Let $\rk(E)=r+1$. Using~\cite[Remark~3.2.3 (b)]{85k:14004},
\begin{equation}\label{eq:ideE}
\frac{c(\pi^*E^\vee\otimes \cO(1))}{c(\cO(1))} = 
c(\cO(1))^r + \sum_{i=1}^r c_i(\pi^*E^\vee) c(\cO(1))^{r-i} 
+ \frac{c_{r+1}(\pi^*E^\vee)}{c(\cO(1))}\quad.
\end{equation}
We have
\[
\pi_*\left(c(\cO(1))^r\cap [\Pbb(E^\vee)]\right)=\pi_*\left(c_1(\cO(1))^r\cap [\Pbb(E^\vee)]\right)=[V]
\quad:
\]
indeed, the other terms in the expansion of $(1+c_1(\cO(1)))^r$ push forward to zero by
\cite[Proposition~3.1(a)(i)]{85k:14004}, and the term $c_1(\cO(1))^r\cap [\Pbb(E^\vee)]$
pushes forward to $[V]$ by~\cite[Proposition~3.1(a)(ii)]{85k:14004}.

The middle term in the right-hand side of~\eqref{eq:ideE} pushes forward to~$0$. Indeed, 
by the projection formula it is a combination of terms
\[
c_i(E^\vee)\cap \pi_*(c_1(\cO(1))^j\cap [\Pbb(E^\vee)])
\]
with $j<r$, and $\pi_*(c_1(\cO(1))^j\cap [\Pbb(E^\vee)])\in A_{\dim V+r-j}(V)=(0)$ for $j<r$.

As for the last term in~\eqref{eq:ideE}, recall that
\[
\pi_*\left(c(\cO(-1))^{-1}\cap [\Pbb(E^\vee)]\right) = c(E^\vee)^{-1}\cap [V]\quad:
\]
indeed, this is essentially the definition of Chern class of a vector bundle according
to~\cite[\S3.2]{85k:14004}. It follows that
\[
\pi_*\left(c(\cO(1))^{-1}\cap [\Pbb(E^\vee)]\right) = (-1)^r c(E)^{-1}\cap [V]\quad,
\]
and therefore
\begin{align*}
\pi_*\left(\frac{c_{r+1}(\pi^*E^\vee)}{c(\cO(1))}\cap[\Pbb(E^\vee)]\right)
&=\pi_*\left((-1)^{r+1} c_{r+1}(\pi^*E) c(\cO(1))^{-1}\cap[\Pbb(E^\vee)]\right) \\
&=-c_{r+1}(E) c(E)^{-1}\cap [V]\quad,
\end{align*}
again by the projection formula.
\end{proof}

The computation of the push-forward of $\cvir(\cX)$ follows from this lemma.

\begin{prop}\label{prop:cvir}
Let $V$ be a nonsingular variety, and $X\subseteq V$ the zero-scheme of a section of a
vector bundle $E$ of rank $r+1$ on $V$. Let $\cV=\Pbb(E^\vee)$,
and let $\cX$ be the hypersurface of $\cV$ defined in~\S\ref{ss:alter}.
Then 
\[
\pi_*(\cvir(\cX))= r\cdot \csm(V) + \frac{c(TV)}{c(E)}\, c_\text{top}(E)\cap [V]
\]
in $A_*V$.
\end{prop}

\begin{proof}
By definition, $\cvir(\cX)=c(T\cV) c(\cO(X))^{-1}\cap [V] = \frac{c(T\cV)}{1+\cX}\cap [\cX]$;
we implicitly view this as a class in $A_*\cV$. By the Euler sequence, the Chern class of the
relative tangent bundle of $\cV=\Pbb(E^\vee)$ is given by $c(T_{\cV/V})=c(E^\vee\otimes \cO(1))$;
therefore
\begin{equation}\label{eq:cTcV}
c(T\cV) = \pi^* c(TV) c(E^\vee\otimes \cO(1))\quad.
\end{equation}
Further, by the normalization and covariance of CSM classes (see~\eqref{eq:cova}),
\[
\pi_* (c(T\cV)\cap [\cV])=\pi_* \csm(\cV) = (r+1) \csm(V)\quad.
\]
(Exercise: Prove this from~\eqref{eq:cTcV}, without using covariance of CSM classes.)
Using these facts and Lemma~\ref{lem:pfE1}:
\begin{align*}
\pi_* \bigg( \frac{c(T\cV)}{1+\cX}\cap [\cX]&\bigg)
=\pi_*\left( c(T\cV)\cap \left([\cV]-\frac 1{1+\cX}\cap [\cV]\right)\right) \\
&=\pi_*\left(c(T\cV)\cap [\cV]\right) - c(TV)\cap \pi_*\left(c(T_{\Pbb(E^\vee)/V}) 
\frac 1{1+\cX}\cap [\cV]\right) \\
&=(r+1)\, c(TV)\cap [V] - c(TV)\cap \pi_*\left(\frac{c(E^\vee\otimes \cO(1))}{c(\cO(1))}\cap [\cV]\right) \\
&=r\cdot c(TV)\cap [V] + \frac{c(TV)}{c(E)}\, c_\text{top}(E)\cap [V]
\end{align*}
as stated.
\end{proof}

\subsection{}\label{ss:calc}
Proposition~\ref{prop:cvir} computes the push-forward of $\cvir(\cX)$. Using the covariance
of CSM classes, it is straightforward to obtain the push-forward of $\csm(\cX)$.

\begin{prop}\label{prop:csm}
Let $V$ be a nonsingular variety, and $\iota: X\hookrightarrow V$ the zero-scheme of a 
section of a vector bundle $E$ of rank $r+1$ on $V$. Let $\cV=\Pbb(E^\vee)$,
and let $\cX$ be the hypersurface of $\cV$ defined in~\S\ref{ss:alter}.
Then 
\[
\pi_*(\csm(\cX))= r\cdot \csm(V) + \iota_* \csm(X)
\]
in $A_*V$.
\end{prop}

\begin{proof}
By definition of CSM class and by covariance,
\[
\pi_*(\csm(\cX)) = \pi_* \csm(\one_{\cX}) = \csm( \pi_* \one_{\cX})\quad.
\]
Now recall (\S\ref{ss:CSM}) that $\pi_*\one_{\cX}$ is the function assigning to $p\in V$ the
Euler characteristic of the fiber of $\cX$ over $p$. 
Use notation as above; in particular, $s:V\to E$ is the section defining $X$. If $p\in X$, 
then $s(p)=0$, and it follows that $s_\cX=\pi^*s\circ\epsilon\equiv 0$ along $\pi^{-1}(p)$.
That is, the fiber of $\cX$ over $p\in X$ equals the fiber of $\cV=\Pbb(E^\vee)$, so it
is an $r$-dimensional projective space. Therefore
\begin{equation}\label{eq:pinX}
p\in X \implies \pi_*(\one_\cX)(p) = \chi(\Pbb^r) = r+1\quad.
\end{equation}
If $p\not\in X$, then $s(p)\ne 0$; $\pi^* s$ is then a fixed vector $(a_0,\dots,a_r)$ of 
$E_p$ along the fiber $\pi^{-1}(p)$. The vanishing of $s_{\cX}$ at $(e_0:\cdots:e_r)
\in\pi^{-1}(p)$ is then equivalent to the linear equation
\[
a_0 e_0+\cdots + a_r e_r=0\quad.
\]
It follows that the fiber of $\cX$ over $p\not\in X$ is a hyperplane $\Pbb^{r-1}$ in the
fiber $\pi^{-1}(p)\cong \Pbb^r$. Therefore
\begin{equation}\label{eq:pnotinX}
p\not\in X \implies \pi_*(\one_\cX)(p) = \chi(\Pbb^{r-1}) = r\quad.
\end{equation}
Combining~\eqref{eq:pinX} and~\eqref{eq:pnotinX}, we obtain that
\[
\pi_*\one_{\cX} = r\cdot \one_V+\one_X\quad,
\]
and the covariance of CSM classes concludes the proof.
\end{proof}

\subsection{}\label{ss:fine}
After these preliminaries we are ready to prove the main result.

\begin{proof}[Proof of Theorem~\ref{thm:main}]
Applying~\cite[Theorem~I.4]{MR2001i:14009} to the hypersurface~$\cX$ gives
\[
\csm(\cX) = \cvir(\cX) + c(T\cV)c(\cO(1))^{-1}\cap \left(s(\cY,\cV)^\vee\otimes_\cV \cO(1)\right).
\]
Here we used the fact that $\cvir(\cX)=c(T\cV)c(\cO(\cX))^{-1}\cap [X] = c(T\cV)\cap s(\cX,\cV)$,
and the fact that $\cY$ is the singularity subscheme of $\cX$, proven in~Lemma~\ref{lem:singy}.
Pushing forward to~$V$ and using Propositions~\ref{prop:cvir} and~\ref{prop:csm}:
\begin{align*}
r &\cdot \csm(V) + \iota_*\csm(X)\\
& =\pi_*\csm(\cX)\\
&=\pi_*\left(\cvir(\cX)+c(T\cV)c(\cO(1))^{-1}\cap \left(s(\cY,\cV)^\vee\otimes_\cV \cO(1)\right)\right)\\
&=r\cdot \csm(V) + \frac{c(TV)}{c(E)}\, c_\text{top}(E)\cap [V]
+\pi_*\left(
c(T\cV)c(\cO(1))^{-1}\cap \left(s(\cY,\cV)^\vee\otimes_\cV \cO(1)\right)
\right)\quad.
\end{align*}
Therefore
\[
\iota_*\csm(X) - \frac{c(TV)}{c(E)}\, c_\text{top}(E)\cap [V]
=\pi_*\left(
\frac{c(T\cV)}{c(\cO(1))}\cap \left(s(\cY,\cV)^\vee\otimes_\cV \cO(1)\right)
\right)\quad.
\]
The statement of Theorem~\ref{thm:main} follows by applying~\eqref{eq:cTcV} and
the projection formula.
\end{proof}

\subsection{}\label{ss:compa}
Using the terminology of Milnor classes, Propositions~\ref{prop:cvir} and~\ref{prop:csm}
immediately imply the following statement.

\begin{prop}\label{prop:milnor}
With notation as above,
\[
\pi_*\cM(\cX) = (-1)^{\dim \cX}\left(
\frac{c(TV)}{c(E)}\, c_\text{top}(E)\cap [V] -\iota_* \csm(X)
\right)
\]
in $A_*V$.
\end{prop}

In particular, paying careful attention to the signs gives:

\begin{corol}\label{corol:milnor}
Assume $\iota: X\hookrightarrow V$ is a local complete intersection, defined as the 
zero-scheme of a regular section of a bundle of rank $\codim_XV$. Then
\begin{equation}\label{eq:equa}
\pi_*\cM(\cX) =\iota_* \cM(X)
\end{equation}
in $A_*V$.
\end{corol}

In the case of local complete intersections, the hypersurface $\cX$ was introduced 
in~\cite{miclalci}. In fact, 
in~\cite[Theorem~6.4]{miclalci}, Callejas-Bedregal, Morgado, and Seade obtain a {\em different\/}
expression relating the Milnor classes of $\cX$ and $X$ in the local complete intersection
case. Comparing~\eqref{eq:equa} and the expression from~\cite{miclalci} may lead to nontrivial 
identities for 
Chern classes of bundles associated with local complete intersections. It would be interesting 
to explore these consequences.


\section{CSM from Segre zeta functions}\label{s:zeta}

\subsection{}\label{ss:zeta}
The results proven in this note draw a direct bridge between Segre classes and CSM classes.
This should allow us to transfer information between these two notions; known facts about
Segre classes should tell us something about CSM classes. This section is an example of
this transfer.

It is known (\cite{MR3709134}) that Segre classes of subschemes of projective space admit
the following description. Let $f_0,\dots, f_m$ be forms of degrees $a_0,\dots, a_m$ 
respectively, in variables $x_0,\dots, x_n$. For $N\ge n$, let $\iota_N: Z_N\hookrightarrow 
\Pbb^N$ be the subscheme defined by the ideal $(f_0,\dots, f_m)$. Then there exists a rational 
function
\[
\zeta(t)=\frac{P(t)}{(1+a_0t)\cdots (1+a_mt)}\quad,
\]
with $P(t)$ a polynomial with nonnegative coefficients and leading term 
$a_0\cdots a_m t^{m+1}$, such that
\[
\iota_{N*} s(Z_N,\Pbb^N) = \zeta(H)\cap [\Pbb^N]\quad.
\]
Here $H$ denotes the hyperplane class. We call $\zeta(t)$ the `Segre zeta function'
determined by the forms $f_0,\dots, f_m$.

A version of this result holds for subschemes of {\em products\/} of projective spaces.
Let $\varphi_0,\dots,\varphi_m$ be bihomogeneous polynomials of 
bidegrees $(a_i,b_i)$, $i=0,\dots, m$, in variables $x_0,\dots, x_n$, $y_0,\dots, y_r$. 
For $N\ge n$, $R\ge r$, let $\iota_{N,R}:Z_{N,R}\hookrightarrow \Pbb^N\times \Pbb^R$ 
be the subscheme defined by the ideal $(\varphi_0,\dots, \varphi_m)$. Then there exists 
a rational function
\[
\zeta(t,u)=\frac{P(t,u)}{(1+a_0t+b_0u)\cdots (1+a_mt+b_mu)}\quad,
\]
with $P(t,u)$ a polynomial with leading term $\prod_i (a_i t + b_i u)$, such that 
\[
\iota_{N,R*} s(Z_{N,R},\Pbb^N\times \Pbb^R) = \zeta(H,h)\cap [\Pbb^N\times \Pbb^R]
\]
where $H$, resp., $h$ denotes the pull-back of the hyperplane class from $\Pbb^N$, 
resp., $\Pbb^R$.

\subsection{}\label{ss:statt}
Theorem~\ref{thm:mainpr} may be expressed in terms of these two-variable zeta functions. 
In fact, we are going to obtain CSM classes directly in terms of the numerator of the zeta 
function determined by the bihomogeneous polynomials defining the scheme $\cY$. This
may simplify the application of Theorem~\ref{thm:mainpr}, and also has the advantage of 
simultaneously computing the CSM classes of the subschemes $X_N\subseteq \Pbb^N$ 
defined by a choice of forms in $x_0,\dots, x_n$, for all $N\ge n$.
(This information could be assembled in a `Segre-Schwartz-MacPherson zeta function'.)

We use the notation introduced in~\S\ref{ss:projr}: $F_0,\dots, F_r$ are homogeneous
polynomials in $x_0,\dots, x_n$, of a fixed degree $d$; the corresponding subscheme $\cY$
of $\Pbb^n \times \Pbb^r$ is defined by the ideal generated by
\begin{equation}\label{eq:eqnY}
F_0,\,\dots,\, F_r\,;\quad \text{and}\quad y_0 \frac{\partial F_0}{\partial x_i} + \cdots +
y_r \frac{\partial F_r}{\partial x_i}\quad,\quad i=0,\dots, n\quad.
\end{equation}
The bidegrees of the generators are $(d,0)$, $(d-1,1)$;
some of the generators may vanish, in which case we view $0$ as a form of the corresponding 
(bi)degree. The chosen generators determine a Segre zeta function for $\cY$:
\begin{equation}\label{eq:zetaY}
\zeta(t,u) = \frac{P(t,u)}{(1+dt)^{r+1}(1+(d-1)t+u)^{n+1}}\quad.
\end{equation}
Therefore, we obtain a well-defined polynomial $P(t,u)\in \Zbb[t,u]$. This polynomial
has degree $n+r+2$, and its term of highest degree is $(dt)^{r+1}((d-1)t+u)^{n+1}$.
As we will see (Remarks~\ref{rem:refin1},~\ref{rem:refin2}), it is actually not necessary to 
know all terms of the polynomial $P(t,u)$ in order to apply the following result: the
terms of degree $\le n+1$ in $t$ and $\le r+1$ in $u$ suffice, and these are determined
by the Segre class of the subscheme $\cY_{n+1,r+1}$ of $\Pbb^{n+1}\times \Pbb^{r+1}$ 
defined by the ideal generated by the forms listed in~\eqref{eq:eqnY}.

\begin{theorem}\label{thm:zeta}
For $N\ge n$, let $\iota_N: X_N\hookrightarrow \Pbb^N$ be the subscheme defined
by the degree-$d$ forms $F_0,\dots, F_r\in k[x_0,\dots, x_n]$. 
With notation as above, let $\gamma(t)$ be the coefficient of $u^{r+1}$ in the polynomial
\[
Q(t,u):=(1+dt+u)^{n+r+2}\cdot P\left(\frac{-t}{1+dt+u}, \frac{-u}{1+dt+u}\right)\quad.
\]
Then 
\[
\iota_{N*} \csm(X_N) = (1+H)^{N-n}\gamma (H)\cap [\Pbb^N]\quad,
\]
where $H$ is the hyperplane class in $\Pbb^N$.
\end{theorem}

\begin{remark}\label{rem:invo}
The transformation
\[
(t,u) \mapsto \left(\frac{-t}{1+dt+u}, \frac{-u}{1+dt+u}\right)
\]
is an involution, and sends $(1+dt+u)$ to $(1+dt+u)^{-1}$. It follows that the operation 
$P\mapsto Q$ defined in the statement of Theorem~\ref{thm:zeta} is an involution.
\qede\end{remark}

\begin{remark}\label{rem:refin1}
It will follow from the proof that $Q(t,u)$ is a polynomial of degree $r+1$ in $u$;
this does not appear to be evident from the definition given in the statement. Thus,
$\gamma(t)$ is actually the leading coefficient of $Q(t,u)$ viewed as a polynomial in $u$.
Since terms in $P(t,u)$ of degree $>r+1$ in $u$ do not contribute to the coefficient of 
$u^{r+1}$ in $Q(t,u)$, $\gamma(t)$ is in fact determined by the the terms of $P(t,u)$ of 
degree $\le r+1$ in $u$.
\qede\end{remark}

\begin{remark}\label{rem:refin2}
Since $Q(t,u)$ has degree $r+1$ in $u$, 
\[
v^{r+1} Q\left(t,\frac 1v\right)
\]
is a polynomial and $\gamma(t)$ is its constant term w.r.t.~$v$. Applying the involution,
\[
v^{r+1} Q\left(t,\frac 1v\right) = \frac {(1+v+dtv)^{n+r+2}}{v^{n+1}} 
P\left(\frac{-tv}{1+v+dtv}, \frac{-1}{1+v+dtv}\right)\quad,
\]
and therefore
\begin{equation}\label{eq:altezeta}
P\left(\frac{-tv}{1+v+dtv}, \frac{-1}{1+v+dtv}\right) =\gamma(t)\cdot  v^{n+1} +
\text{higher order terms in $v$}\quad.
\end{equation}

This gives an alternative computation of the term $\gamma(t)$ obtained in Theorem~\ref{thm:main}.
It also shows that the terms of $P(t,u)$ of degree $>n+1$ in $t$ do not affect $\gamma(t)$.
(However, note that~\eqref{eq:altezeta} may be affected by terms of $P(t,u)$ of degree $\ge r+1$ 
in $u$. This limits its applicability.)

Summarizing, only the terms of $P(t,u)$ of degrees $\le n+1$ in $t$ and $\le r+1$ in $u$
are needed in order to apply Theorem~\ref{thm:zeta}.
\end{remark}

\subsection{}
The proof of Theorem~\ref{thm:zeta} will use the following simple observation, 
for which we do not have a reference.

\begin{lemma}\label{lem:sta}
Let $S(h)$ be a power series with coefficients in a ring. Assume that the coefficient $C$
of $h^R$ in $(1+h)^R\cdot S(h)$ is nonzero and independent of $R$ for $R\ge N$. 
Then $(1+h)^N S(h)$ is a polynomial of degree $N$ in~$h$, with leading coefficient $C$.
\end{lemma}

\begin{proof}
The coefficient of $h^N$ in $(1+h)^N S(h)$ is $C$ by hypothesis.
Arguing by contradiction, assume that $(1+h)^N S(h)$ is {\em not\/} a polynomial
of degree $N$; then it must have a first nonzero term $s_M h^M$ with $M>N$.
Note that then
\[
(1+h)^{M-1} S(h) = s_0 + \cdots + C h^{M-1} + s_M h^M + \cdots\quad.
\]
It follows that 
\[
(1+h)^M S(h) = s_0 + \cdots + (C+s_M) h^M + \cdots\quad,
\]
so the coefficient of $h^M$ in $(1+h)^M S(h)$ is $C+s_M\ne C$, contrary to 
the hypothesis.
\end{proof}

We are now ready to prove Theorem~\ref{thm:zeta}. Its derivation from 
Theorem~\ref{thm:mainpr} is a good exercise in the use of the properties of the 
notation $\otimes$, ${}^\vee$.

\begin{proof}[Proof of Theorem~\ref{thm:zeta}]
Since the ideal of $X_N$ is generated by $F_0,\dots, F_r$, the corresponding subscheme
$\iota:\cY_{N,R}\hookrightarrow \cV_{N,R}:=\Pbb^N\times \Pbb^R$ is defined by the forms listed 
in~\eqref{eq:eqnY}, for all $N\ge n$ and $R\ge r$. (We may choose $F_{r+1}=\cdots=F_R=0$.) 
Therefore the push-forward of the Segre class of $\cY_{N,R}$ is given by the zeta 
function~\eqref{eq:zetaY}:
\[
\iota_* s(\cY_{N,R},\cV_{N,R}) = \left(\frac{P(H,h)}{(1+dH)^{r+1} (1+(d-1)H+h)^{n+1}}\right)
\cap [\Pbb^N\times \Pbb^R]\quad.
\]
By properties of $\otimes$, ${}^\vee$ from~\cite[Proposition~1]{MR96d:14004},
\begin{align*}
\iota_* s(\cY_{N,R},\cV_{N,R})^\vee\otimes_{\cV_{N,R}} &\cO(dH+h)
=\frac{P(-H,-h)}{(1-dH)^{r+1} (1-(d-1)H-h)^{n+1}}\otimes_{\cV_{N,R}} \cO(dH+h) \\
&=\frac{(1+dH+h)^{n+r+2}}{(1+h)^{r+1} (1+H)^{n+1}} 
\left(P(-H,-h)\otimes_{\cV_{N,R}} \cO(dH+h)\right)\\
&=\frac{(1+dH+h)^{n+r+2}P\left(\frac{-H}{1+dH+h},\frac{-h}{1+dH+h}\right)}{(1+h)^{r+1} (1+H)^{n+1}} 
\quad.
\end{align*}
Let $Q(t,u)$ be the polynomial $(1+dt+u)^{n+r+2}P\left(\frac{-t}{1+dt+u},\frac{-u}{1+dt+u}\right)$.
By Theorem~\ref{thm:mainpr}, $\iota_{N*}\csm(X_N)$ is the push-forward of
\begin{multline*}
\frac{(1+H)^{N+1} (1+h)^{R+1}}{1+dH+h} 
\left(s(\cY, \Pbb^n\times \Pbb^r)^\vee \otimes_{\Pbb^n\times \Pbb^r}\cO(dH+h)\right) \\
=(1+H)^{N-n} (1+h)^R \frac{Q(H,h)}{(1+h)^r (1+dH+h)}
\cap [\Pbb^N\times \Pbb^R]\quad,
\end{multline*}
provided $R\ge N$. The push-forward is obtained by capping against $[\Pbb^N]$
the coefficient of~$h^R$ in
\[
(1+H)^{N-n} (1+h)^R \frac{Q(H,h)}{(1+h)^r (1+dH+h)}\quad.
\]
We view this expression as a power series in $h$ with coefficients in $A_*(\Pbb^N)$,
and note that the coefficient of $h^R$ is $\iota_{N*}\csm(X_N)$, {\em independently
of $R\ge N$.\/} By Lemma~\ref{lem:sta}, 
\[
(1+H)^{N-n} (1+h)^N \cdot \frac{Q(H,h)}{(1+h)^r (1+dH+h)}
\]
is a {\em polynomial\/} in $h$ with coefficients in $A_*(\Pbb^N)=\Zbb[H]/(H^{N+1})$, 
of degree $N$ and leading coefficient $\iota_{N*}\csm(X_N)$. It then follows that
\[
(1+H)^{N-n} Q(H,h)
\]
is a polynomial of degree $r+1$ in $h$, with leading coefficient $\iota_{N*}\csm(X_N)$,
and this is the statement. 

Note that this argument shows that $Q(t,u)$ is a polynomial of degree $r+1$ in $u$ modulo
$t^{N+1}$ for every $N\gg 0$. It follows that it has degree $r+1$ in $u$ as a polynomial
in $\Zbb[t,u]$.
\end{proof}

\subsection{}
We give two examples illustrating Theorem~\ref{thm:zeta}.

\begin{example}
Consider the forms $F_0=x_1 x_2$, $F_1=x_0x_2$, $F_2=x_0x_1$. (Thus $n=r=2$.)
The corresponding scheme in $\Pbb^N$ consists of the union of three codimension~$2$
subspaces meeting along a common codimension~$3$ subspace. 

The generators of the ideal of $\cY$ in this example are
\[
x_1 x_2\,,\, x_0x_2\,,\, x_0 x_1\,;\, x_2 y_1 + y_2 x_1\,,\, y_0 x_2+ y_1 x_2\,,\, y_0 x_1+y_1 x_0
\quad.
\]
We have
\[
\zeta(t,u)=\frac{P(t,u)}{(1+2t)^3(1+t+u)^3}\quad;
\]
the requirement that $\zeta(t,u)$ evaluates the Segre class in $\Pbb^3\times \Pbb^3$
determines the terms of $P(t,u)$ of degree $\le n+1=3$ in $t$ and $\le r+1=3$ in $u$.
With the aid of the Macaulay2 package~\cite{HJ} we get
\[
P(t,u) = t^3+6t^3u+3t^2u^2+18 t^3u^2+3t^2u^3+8t^3u^3
+\text{higher order terms}\quad.
\]
The polynomial $Q$ appearing in the statement of Theorem~\ref{thm:zeta} is therefore
\[
(1+2t+u)^6\cdot P\left(\frac{-t}{1+2t+u}, \frac{-u}{1+2t+u}\right)
\equiv -t^3+3t^3u+(3t^2+3t^3)u^2+(3t^2+t^3)u^3\mod t^4
\]
This is necessarily a polynomial of degree $r+1=3$ in $u$ (Remark~\ref{rem:refin1}),
and the coefficient of $u^3$ is $3t^2+t^3$. By Theorem~\ref{thm:zeta}, we can conclude 
that
\begin{equation}\label{eq:u3ex}
\iota_{N*}\csm(X_N) = (1+H)^{N-2} (3H^2+H^3) \cap [\Pbb^N]\quad.
\end{equation}
For example, for $N=6$ this gives
\[
\iota_{6*}\csm(X_6) = 3[\Pbb^4]+13 [\Pbb^3]+22[\Pbb^2]+18[\Pbb^1]+7[\Pbb^0]\quad.
\]
The reader may enjoy verifying independently that~\eqref{eq:u3ex} holds, by using 
the geometric description of $X_N$ given at the beginning of this example.
\qede\end{example}

\begin{example}\label{ex:last}
In closing, we revisit the example given in~\S\ref{ss:exa}, consisting of the complete
intersection of $x_1 x_2 x_3=0$ and $x_0x_1^2+x_2^3=0$. Here $n=3,r=1$, so the
needed information can be extracted from the Segre class of the subscheme 
defined by
\[
x_1 x_2 x_3\,,\, x_0x_1^2+x_2^3\,;\,
y_1 x_1^2\,,\, y_0x_2x_3+2y_1x_0x_1\,,\, y_0x_1x_3+3y_1x_2^2\,,\, y_0x_1 x_2
\]
in $\Pbb^4\times \Pbb^2$. According to~\cite{HJ}, this Segre
class is
\[
(H^2+3H^2h+H^3+2H^2h^2-19H^3h-16H^4-30H^3h^2+69H^4h+240H^4h^2)\cap
[\Pbb^4\times \Pbb^2]
\]
where $H$, resp.~$h$ is the pull-back of the hyperplane class from the first, resp.~second
factor. It follows that the polynomial $P(t,u)$ corresponding to this example must be
\[
P(t,u) = t^2+15t^3+79t^4+(7t^2+75t^3+258t^4)u+(20t^2+132t^3+216t^4)u^2
+\text{higher order terms.}
\]
Applying the involution defined in the statement of Theorem~\ref{thm:zeta} gives
\[
Q(t,u) \equiv (t^2-3t^3-2t^4)+(-3t^2+3t^3-t^4)u+(5t^2+3t^3+t^4)u^2 \mod t^5
\]
and we can conclude that the CSM class of the complete intersection defined by
$x_1 x_2 x_3=0$ and $x_0x_1^2+x_2^3=0$ in $\Pbb^N$ pushes forward to
\[
(1+H)^{N-3}(5H^2+3H^3+H^4)\cap [\Pbb^N]\quad.
\]
The reader can verify that for $N=6$ this is in agreement with the result obtained in~\S\ref{ss:exa}.
\qede\end{example}


\end{document}